\documentclass{article}

\usepackage{amsthm}
\usepackage{epsfig}
\newtheorem{thm}{Theorem}

\newtheorem{prp}{Proposition}
\newtheorem{cor}{Corollary}

\def\rank{{\mbox{rank}}}
\def\mod{{\mbox{mod}}}

\def\F{\bf{F}}
\def\h{{\bf h}}
\def\k{{\bf k}}
\def\0{{\bf 0}}
\def\1{{\bf 1}}

\title{Spectral symmetry in conference matrices}

\author{Willem H. Haemers\thanks{corresponding author; e-mail haemers@uvt.nl}
\\
{\it\small
Department of Econometrics and Operations Research, Tilburg University,}
\\
{\it\small  Tilburg, The Netherlands}
\\[5pt]
Leila Parsaei Majd\thanks{e-mail leila.parsaei@ipm.ir}
\\
{\it\small
School of Mathematics, Institute for Research in Fundamental Sciences (IPM),}
\\
{\it\small
P.O. Box 19395-5746, Tehran, Iran}
}
\date{}

\begin{document}
	
\maketitle

\begin{abstract}
A conference matrix of order $n$ is an $n\times n$ matrix $C$ with diagonal entries $0$ and off-diagonal entries $\pm 1$ satisfying
$CC^\top=(n-1)I$.
If $C$ is symmetric, then $C$ has a symmetric spectrum $\Sigma$ (that is, $\Sigma=-\Sigma$) and eigenvalues $\pm\sqrt{n-1}$.
We show that many principal submatrices of $C$ also have symmetric spectrum, which leads to
examples of Seidel matrices of graphs (or, equivalently, adjacency matrices of complete signed graphs) with a
symmetric spectrum.
In addition, we show that some Seidel matrices with symmetric spectrum can be characterized by this construction.
\\[3pt]
\end{abstract}

\section{Introduction}
Suppose $S$ is a symmetric matrix with zero diagonal and off-diagonal entries $0$ or $\pm 1$.
Then $S$ can be interpreted as the adjacency matrix of a signed graph.
Signed graphs are well studied, and a number of recent papers \cite{AMP, BCKW, GHMP, R} pay attention to signed graphs for which
the adjacency matrix has symmetric spectrum,
which means that the spectrum is invariant under multiplication by $-1$.
If $S$ contains no $-1$, then $S$ is the adjacency matrix of an ordinary graph,
which has a symmetric spectrum if and only if the graph is bipartite.
For general signed graphs there exist many other examples with symmetric spectrum.
Here we consider the case that no off-diagonal entries are $0$, in which case $S$ can be interpreted as the Seidel matrix of a graph
($-1$ is adjacent; $+1$ is non-adjacent).
It is known that a Seidel matrix of order $n$ is nonsingular if $n \not\equiv 1~(\mod~4)$ (see Greaves~at~al~\cite{GKMS}).
Clearly, a symmetric spectrum contains an eigenvalue $0$ if $n$ is odd,
therefore there exists no Seidel matrix with symmetric spectrum if $n\equiv 3~(\mod~4)$.
For all other orders Seidel matrices with spectral symmetry exist.
Examples are often built with smaller block matrices.
Here we use a different approach, and investigate Seidel matrices with spectral symmetry inside larger matrices known as conference matrices
(see next section).

The spectrum of $S$ does not change if some rows and the corresponding columns are multiplied by $-1$.
This operation is called switching.
If ${S}'$ can be obtained from $S$ by switching and/or reordering rows and columns, $S$ and ${S}'$ are called equivalent.
The corresponding graphs are called switching isomorphic, or switching equivalent.

\section{Conference matrices}\label{conf}

A conference matrix of order $n$ ($n\geq 2$) is an $n\times n$ matrix $C$ with diagonal entries $0$
and off-diagonal entries $\pm 1$ satisfying $CC^\top=(n-1)I$.
If $C$ is symmetric, then the spectrum $\Sigma$ of $C$ contains the eigenvalues $\sqrt{n-1}$ and $-\sqrt{n-1}$, both with multiplicity $n/2$,
and we write:
\[
\Sigma=\{\pm\sqrt{n-1}^{\, n/2}\}.
\]
Clearly
the spectrum of $C$ is symmetric.
Conference matrices are well studied (see for example Section~13 of~Seidel~\cite{S}, and Section~10.4 of~Brouwer and Haemers~\cite{BH}).
The order $n$ of a conference matrix is even, and every conference matrix can be switched into a symmetric one when $n\equiv 2~(\mod~4)$
and into a skew-symmetric one if $n\equiv 0~(\mod~4)$.
Here we will not consider the skew case, because every skew-symmetric matrix has a symmetric spectrum.
Necessary for the existence of a symmetric conference matrix of order $n$ is that $n-1$ is the sum of two squares.

If $C$ is a symmetric conference matrix of order $n=4m+2\geq 6$, switched such that all off-diagonal entries in the first row and column are
equal to $1$, then $C$ is the Seidel matrix of a graph with an isolated vertex.
If we delete the isolated vertex, we have a strongly regular graph $G$ with parameters $(4m+1,2m,m-1,m)$
(this means that $G$ has order $4m+1$, is $2m$-regular, every edge is in precisely $m-1$ triangles,
and any two nonadjacent vertices have precisely $m$ common neighbors).
Conversely, the Seidel matrix of a strongly regular graph with these parameters extended with an isolated vertex,
is a symmetric conference matrix.
If $4m+1$ is a prime power $q$ (say), such a strongly regular graph can be constructed as follows.
The vertices of $G$ are the elements of the finite field $\F_q$,
where two vertices $x$ and $y$ $(x\neq y)$ are adjacent whenever $x-y$ is a square in $\F_q$.
The construction is due to Paley, the graph $G$ is known as Paley graph, and a corresponding conference matrix
$C$ is a Paley conference matrix of order $n=q+1$, which we shall abbreviate to $PC(n)$.
Other constructions are known.
Mathon~\cite{M} has constructed conference matrices of order $n=q p^2+1$, where
$p$ and $q$ are prime powers, $q\equiv 1~(\mod~4)$, and $p\equiv 3~(\mod~4)$.
For $n=6$, 10, 14 and $18$, every conference matrix of order $n$ is a $PC(n)$.
There is no conference matrix of order $22$ ($21$ is not the sum of two squares),
and there are exactly four non-equivalent conference matrices of order 26, one of which is a $PC(26)$.
The smallest order for which existence is still undecided is 86.

\section{The tool}

\begin{thm}\label{main}
Suppose
\[
A=
\left[
\begin{array}{cc}
A_1 & M\\
M^\top & A_2
\end{array}
\right]
\]
is a symmetric orthogonal matrix with zero diagonal.
Let $n_i$ be the order, and let $\Sigma_i$ be the spectrum of $A_i$ ($i=1,2$).
Assume $n_1\leq n_2$, and define $m=(n_2-n_1)/2$.
Then
\[
\Sigma_2 = -\Sigma_1 \cup \{\pm 1^m\}.
\]
\end{thm}

\noindent
(Recall that $\{\pm 1^m\}$ means that $1$ and $-1$ are both repeated $m$ times.)

\begin{proof}
We have
\[
I = A^2=
\left[
\begin{array}{cc}
A_1^2+M M^\top & A_1 M+M A_2\\
M^\top A_1 + A_2 M^\top & A_2^2 + M^\top\!M
\end{array}
\right].
\]
This implies $A_1 M = -M A_2$, $A_1^2+MM^\top\! = I$ and $A_2^2 + M^\top\! M = I$.
For $i=1,2$ let $\Sigma_i'$ be the sub-multiset of $\Sigma_i$ obtained by deleting all eigenvalues equal to $\pm 1$.
Suppose $\lambda\in\Sigma_1'$ is an eigenvalue of $A_1$ with multiplicity $\ell$.
Define $V$ such that its columns span the eigenspace of $\lambda$.
Then $A_1 V=\lambda V$, and rank$(V)=\ell$.
Moreover, $\lambda M^\top\! V = M^\top\! A_1 V = -A_2 M^\top\! V$.
Therefore $-\lambda$ is an eigenvalue of $A_2$ and the columns of $M^\top\! V$ are eigenvectors.
Using $A_1^2+MM^\top\! = I$ and $\lambda\neq\pm 1$ we have
\[
\rank(V) \geq \rank(M^\top\! V) \geq \rank(MM^\top\! V) = \rank((I-A_1^2)V) = \rank((1-\lambda^2)V) = \rank(V).
\]
Therefore $\rank(M^\top V) = \rank(V)$, and the multiplicity ${\ell}'$  of $-\lambda\in\Sigma_2'$ is at least $\ell$.
Conversely, $\ell\geq{\ell}'$ and therefore $\ell = {\ell}'$.
This implies that $\Sigma_1' = -\Sigma_2'$.
Finally, trace$(A_1) = \mbox{trace}(A_2)=0$ implies that for both matrices the eigenvalues $-1$ and $1$ have the same multiplicities.
\end{proof}	

\begin{cor}\label{cor}
Suppose
\[
C=
\left[
\begin{array}{cc}
C_1 & N\\
N^\top & C_2
\end{array}
\right]
\]
is a symmetric conference matrix of order $n$.
\\
$(i)$ ~$C_1$ has a symmetric spectrum if and only if $C_2$ has a symmetric spectrum.
\\
$(ii)$~If $C_1$ and $C_2$ have symmetric spectrum then, except for eigenvalues equal to $\pm\sqrt{n-1}$,
$C_1$ and $C_2$ have the same spectrum.
\end{cor}

\begin{proof}
Apply Theorem~\ref{main} to $A=\frac{1}{\sqrt{n-1}}C$.
\end{proof}	

Theorem~\ref{main} is a special case of an old tool, which has proved to be useful in spectral graph theory.
It is, in fact, a direct consequence of the inequalities of Aronszajn (see~\cite{H}, Theorem~1.3.3).

\section{Submatrices}

Clearly a Seidel matrix of order 1 or 2 has symmetric spectrum, so by Corollary~\ref{cor} we obtain Seidel matrices with spectra
$\{0,\ \pm\sqrt{n-1}^{\,(n-2)/2}\}$ and $\{\pm 1,\ \pm\sqrt{n-1}^{\,(n-4)/2}\}$
if we delete one or two rows and the corresponding columns from a symmetric conference matrix of order~$n$.
In the next section we will characterise this construction.

As mentioned earlier, there is no Seidel matrix with symmetric spectrum if $n\equiv 3~(\mod~4)$.
For $n=4$ and $5$, there is exactly one equivalence class of Seidel matrices with spectral symmetry, represented by:
\[
S_4=
{\small
\left[
\begin{array}{rrrr}
0&1&1&1\\
1&0&-1&1\\
1&-1&0&-1\\
1&1&-1&0
\end{array}
\right]
},
\mbox{ and }
S_5=
{\small
\left[
\begin{array}{rrrrr}
0&1&1&1&1\\
1&0&-1&1&1\\
1&-1&0&-1&1\\
1&1&-1&0&-1\\
1&1&1&-1&0
\end{array}
\right]
}
\]
with spectra
\[
\{\pm 1,\ \pm\sqrt{5}\} \mbox{ and } \{0,\ \pm\sqrt{5}^2\}.
\]

\begin{prp}
After suitable switching, every symmetric conference matrix of order $n\geq 6$ contains $S_4$ and $S_5$ as a principal submatrix.
\end{prp}

\begin{proof}
The graphs of $S_4$ and $S_5$ have an isolated vertex.
If we delete the isolated vertex we obtain the paths $P_3$ and $P_4$.
So, it suffices to show that a strongly regular graph $G$ with parameters $(4m+1,2m,m-1,m)$ contains $P_3$ and $P_4$ as an induced subgraph.
The presence of $P_3$ in $G$ is trivial.
Fix an edge $\{x,y\}$ in $G$, and let $z$ be a vertex adjacent to $y$, but not to $x$.
Then there are $m$ vertices which are adjacent to $x$ and not to $y$, and at least one of them ($w$ say) is nonadjacent to $z$,
since otherwise $x$ and $z$ would have $m+1$ common neighbors.
Thus the set $\{w,x,y,z\}$ induces a $P_4$.
\end{proof}

By Corollary~\ref{cor} and the above proposition we know that there exist Seidel matrices of order $n-4$ and $n-5$ with spectra
\[\{\pm 1,\ \pm\sqrt{5},\ \pm\sqrt{n-1}^{(n-8)/2}\}, \mbox{ and }
\{0,\ \pm\sqrt{5}^2,\ \pm\sqrt{n-1}^{(n-10)/2}\},
\]
respectively, whenever there exists a symmetric conference matrix of order $n\geq 10$.

Next we investigate how Corollary~\ref{cor} can be applied to a $PC(n)$ for $n=10$, 14, and 18 with a submatrix of order 6, 8, or 9.
Up to equivalence there exist four Seidel matrices of order 6 with symmetric spectrum (see~Van~Lint and Seidel~\cite{LS},
or Ghorbani~et~al~\cite{GHMP}).
The spectra are:
\[
\Sigma_1=\{\pm 1,\ \pm\sqrt{5},\ \pm 3\},~
\Sigma_2=\{\pm\sqrt{5}^3\},~
\Sigma_3=\{\pm 1,\ \pm\sqrt{7\pm 2\sqrt{5}}\},~
\Sigma_4=\{\pm 1^2,\ \pm\sqrt{13}\}.
\]
Only $\Sigma_1$ is the spectrum of a submatrix of a $PC(10)$,
each of the spectra $\Sigma_1,~\Sigma_2$ and $\Sigma_3$ belongs to a submatrix of a $PC(14)$,
and all four occur as the spectrum of a submatrix of a $PC(18)$.
So by Corollary~\ref{cor} we obtain Seidel matrices of order 8 and 12 with spectra
\[
\Sigma_i\cup\{\pm\sqrt{13}\} \mbox{ for } i=1,2,3 \mbox{ and }
\Sigma_i\cup\{\pm\sqrt{17}^3\} \mbox{ for } i=1,\ldots,4,
\]
respectively.
All graphs of order 8 with a symmetric Seidel spectrum are given in Figure~6 of Ghorbani~et~al~\cite{GHMP}.
Up to equivalence and taking complements there are twenty such graphs (we just found three of these).
By computer we found that six of these graphs have a Seidel matrix, which is a submatrix of a $PC(18)$.
So Corollary~\ref{cor} gives six possible symmetric spectra for the graphs on the remaining 10 vertices.
However, it turns out that these six spectra belong to seven non-equivalent graphs, of which two have the same Seidel spectrum.
These seven graphs are given in Figure~\ref{10} (the last seven graphs).
The same phenomenon occurs if we delete $S_4$ from a $PC(14)$.
This can be done in two non-equivalent ways,
which leads to two non-equivalent graphs with spectrum $\{\pm 1,\ \pm\sqrt{5},\ \pm\sqrt{13}{\,}^3\}$ (the first two in Figure~\ref{10}).
The Seidel matrices of order 9 with symmetric spectrum are also given in \cite{GHMP}.
It turns out that none of these is a submatrix of a $PC(18)$.
But note that we already found two non-equivalent Seidel matrices of order 9 with symmetric spectrum,
one in a $PC(10)$ and one in a $PC(14)$.
Also the Seidel matrix of order 8 with spectrum $\{\pm 1,\ \pm 3^3\}$ is a submtrix of a $PC(10)$, but not of a $PC(14)$ or a $PC(18)$.
Similarly, the Seidel matrix of order 8 with spectrum $\Sigma_1\cup\{\pm\sqrt{13}\}$ is a submatrix of a $PC(14)$ (as we saw above),
but not of a $PC(18)$.
\begin{figure}[h]
\epsfig{file = 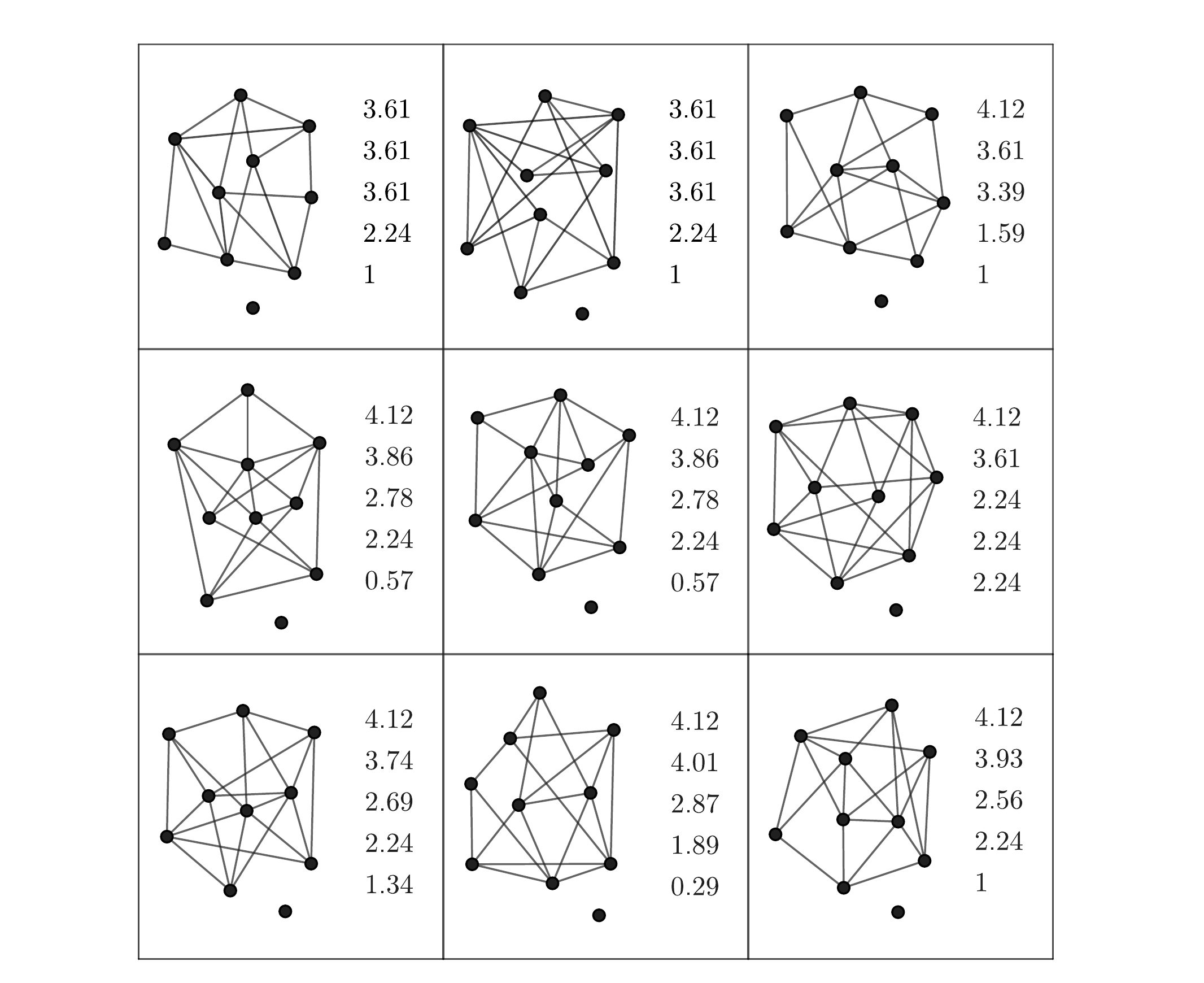,height=300pt,width=400pt}
\caption{Graphs of order $10$ for which the Seidel matrix is a submatrix of a $PC(14)$ (first 2),
or a $PC(18)$ (last 7); the numbers represent the positive part of the (symmetric) spectrum.}\label{10}
\end{figure}
%
\\
It is known (see Bollobas and Thomason~\cite{BT}) that every graph of order $m$ is an induced subgraph of the Paley graph of order
$q$ if $q\geq f(m)=(2^{m-2}(m-1)+1)^2$.
If the smaller graph is the Paley graph of order $m$, it follows that a $PC(m+1)$ is a principal submatrix of a $PC(q+1)$ if $q\geq f(m)$.
Thus, by Corollary~\ref{cor}, we obtain Seidel matrices with symmetric spectrum containing only four distinct eigenvalues:
\begin{prp}\label{P}
If $q$ and $m$ are prime powers satisfying $q\equiv m\equiv 1$~{\rm (mod 4)} and $q\geq f(m)$, then
there exist a Seidel matrix with spectrum $\{\pm\sqrt{q}^{\,(q-2m-1)/2},\ \pm\sqrt{m}^{\,(m+1)/2}\}$.
\end{prp}

\section{Characterizations}

It is clear that a Seidel matrix with symmetric spectrum and two distinct eigenvalues is a conference matrix.
The next two theorems deal with three and four distinct eigenvalues.

In a more general setting, the results of this section were already obtained by Greaves and Suda~\cite{GS}.
In the proofs below we restrict to the case which is relevant to us: spectral symmetry in a symmetric conference matrix.

\begin{thm}\label{char1}
Suppose $S$ is a Seidel matrix with symmetric spectrum and three distinct eigenvalues,
then $S$ can be obtained from a symmetric conference matrix by deleting one row and the corresponding column.
\end{thm}

\begin{proof}
Suppose $S$ has order $n-1$.
It follows that $S$ has an eigenvalue $0$ of multiplicity $1$, and two eigenvalues $\pm\sqrt{n-1}$, each of multiplicity $(n-2)/2$.
Define $M=(n-1)I-S^2$, then rank$(M)=1$, the diagonal entries of $M$ are equal to $1$, and $M$ is positive semi-definite.
This implies that $M=\k\k^\top$ for some vector $\k$ with entries $\pm 1$.
It follows that $\k^\top S^2 \k=\k^\top ((n-1)I-M)\k=(n-1)^2-(n-1)^2=0$, hence $S\k=\0$.
Define
\[
C=\left[
\begin{array}{cc}
0 & \k^\top\\
\k & S
\end{array}
\right].
\mbox{ Then }
CC^\top=C^2=
\left[
\begin{array}{cc}
n & \k^\top\!S\\
S\k & \k\k^\top\!+S^2
\end{array}
\right]
= (n-1)I,
\]
because $S\k=\0$ and $S^2=(n-1)I-M=(n-1)I-\k\k^\top$.
\end{proof}

If $S$ is the Seidel matrix of a regular graph $G$, then $G$ is strongly regular, as we have seen in Section~\ref{conf}.
Here we do not require regularity.
However it follows from the above that we can always switch in $C$, such that $\k$ becomes the all-one vector,
in which case the switched $G$ is regular.

\begin{thm}
Suppose $S$ is a Seidel matrix with symmetric spectrum and four distinct eigenvalues, which include $1$ and $-1$ both of multiplicity $1$,
then $S$ can be obtained from a symmetric conference matrix by deleting two rows and the corresponding columns.
\end{thm}

\begin{proof}
Suppose $S$ has order $n-2$.
Clearly $n$ is even, and $S^2$ has an eigenvalue $1$ of multiplicity $2$.
From trace$(S^2)=(n-2)(n-3)$ it follows that $S^2$ has one other eigenvalue equal to $n-1$ of multiplicity $n-4$.
Define $M=(n-1)I-S^2$.
Then rank$(M)=2$, and $M$ is positive semi-definite with an eigenvalue $n-2$ of multiplicity $2$.
The diagonal entries of $M$ are equal to $2$, and the off-diagonal entries are even integers.
Let $T$ be a pricipal submatrix of $M$ of order $2$, then $T$ is positive semi-definite, and therefore $T$ is one of the following:
\[
\left[
\begin{array}{rr}
2 & -2\\
-2 & 2
\end{array}
\right],
\
\left[
\begin{array}{rr}
2 & 2\\
2 & 2
\end{array}
\right],
\mbox{ or }
\left[
\begin{array}{cc}
2 & 0\\
0 & 2
\end{array}
\right].
\]
Since rank$(M)=2$, $M$ has a $2\times 2$ pricipal submatrix of rank~2, so the last option $T=2I$ does occur.
Consider the two rows in $M$ corresponding to $T=2I$.
At each coordinate place, the two entries can only consist of one $0$ and one $\pm 2$,
since all other options would create a submatrix of $M$ of rank~3.
Every other row of $M$ is a linear combination of these two rows, and because $M$ is symmetric,
we can conclude that the rows and columns of $S$ can be ordered such that
\[
M=2\left[
\begin{array}{cc}
M_1 & O\\
O & M_2
\end{array}
\right],
\]
where $M_1$ and $M_2$ have $1$ on the diagonal, $\pm 1$ off-diagonal, and $\rank(M_1)=\rank(M_2)=1$.
This implies that there exist vectors $\k_1$ and $\k_2$ with entries $\pm 1$, such that
$M_1=\k_1\k_1^\top$ and $M_2=\k_2\k_2^\top$.
Both $M_1$ and $M_2$ have one nonzero eigenvalue $(n-2)/2$, which equals the trace,
so $M_1$ and $M_2$ have the same order $(n-2)/2$.
With the corresponding partition of $S$ we have
\[
S=\left[
\begin{array}{cc}
S_1 & R\\
R^\top & S_2
\end{array}
\right],
\ S^2=\left[
\begin{array}{cc}
S_1^2+RR^\top & S_1 R + RS_2\\
R^\top\!S_1+ S_2 R^\top & R^\top\!R+S_2^2
\end{array}
\right]=(n-1)I-2
\left[
\begin{array}{cc}
M_1 & O\\
O & M_2
\end{array}
\right].
\]
We conclude that $S_1 R=-RS_2$, and $S_1^2+RR^\top=(n-1)I-2M_1$.
Using $\k_1^\top \k_1=(n-2)/2$, and $M_1=\k_1\k_1^\top$ we obtain
\[
\k_1^\top S_1^2\k_1+\k_1^\top RR^\top\k_1 = \k_1^\top(S_1^2+RR^\top)\k_1 = \k_1^\top((n-1)I-2M_1)\k_1=(n-2)/2.
\]
The entries of $S_1\k_1$ are odd integers, so $\k_1^\top S_1^2\k_1\geq(n-2)/2$.
This implies $\k_1^\top RR^\top\k_1=0$ and $\k_1^\top S_1^2\k_1=(n-2)/2$,
so $R^\top\k_1=\0$, and $S_1\k_1$ is a $(\pm 1)$-vector $\h_1$ (say).
Next observe that
$R^\top \h_1 = R^\top S_1\k_1 = -S_2R^\top\k_1=\0$,
and also $S_1\h_1 = S_1^2\k_1 = (-RR^\top + (n-1)I - 2\k_1^\top\k_1)\k_1=\k_1$.
Similarly, $\h_2=S_2\k_2$ is a $(\pm 1)$-vector, $S_2\h_2=\k_2$, and $R\k_2=R\h_2=\0$.
Define
\[
C = \left[
\begin{array}{crrr}
   0 &   1\   & -\h_1^\top & \,\ \h_2^\top \\
   1 &   0\   & -\h_1^\top & -\h_2^\top \\
\k_1 &  \k_1  & S_1~       &   R~~        \\
\k_2 & -\k_2  & R^\top     &   S_2~
\end{array}
\right].
\]
Then $CC^\top=(n-1)I$.
Finally $C^\top C= (n-1)I$ implies that $\h_1=-\k_1$ and $\h_2=\k_2$.
\end{proof}

As we have seen in Proposition~\ref{P},
there exist many Seidel matrices with four distinct eigenvalues and symmetric spectrum different from the ones in the above theorem.
Another example is the Seidel matrix of a complete graph of order $m$, extended with $m$ isolated vertices.
(see Ghorbani~et~al~\cite{GHMP}, Theorem 2.2).
It is not likely that the case of four eigenvalues can be characterised in general.

Note that the above characterizations lead to nonexistence of some Seidel spectra.
For example, there exist no graphs with Seidel spectra $\{0,\ \pm\sqrt{21}{\,}^{10}\}$
and $\{\pm 1,\ \pm\sqrt{21}{\,}^9\}$, because there exist no conference matrix of order $22$.

\section{Sign-symmetry}

A graph $G$ is called sign-symmetric if $G$ is switching isomorphic to its complement.
If $S$ is the Seidel matrix of a sign-symmetric graph $G$ (we also call $S$ sign-symmetric),
then $S$ and $-S$ are equivalent, and therefore $S$ has symmetric spectrum.

Every $PG(n)$ is sign-symmetric, but many other conference matrices are not.
Up to equivalence, there are at least two conference matrices of order 38,
and at least 80 of order 46 which are not sign-symmetric (see Bussemaker, Mathon and Seidel~\cite{BMS}).

If we delete one or two rows and columns from a $PC(n)$, the obtained Seidel matrix will be sign-symmetric.
But in general, when we apply Corollary~\ref{cor},
there is not much we can say about the relation between sign-symmetry of $C$, $C_1$ and $C_2$.

\end{document}